\documentclass[letterpaper, 10 pt, conference]{ieeeconf}  

\IEEEoverridecommandlockouts                              
\usepackage{amsmath,amsfonts}
\usepackage{graphicx}
\usepackage{dsfont}
\usepackage{enumerate}
\usepackage{mathrsfs}
\usepackage{mathabx}
\usepackage{xcolor}

\newtheorem{theorem}{Theorem}[section]

\newtheorem{proposition}[theorem]{Proposition}
\newtheorem{corollary}[theorem]{Corollary}
\newtheorem{definition}[theorem]{Definition}

\newtheorem{remark}[theorem]{Remark}

\title{\LARGE \bf
Gauge Freedom within the Class of Linear Feedback Particle Filters
}

\newcommand{\R}{\mathds{R}}

\newcommand{\Skew}[1]{\text{Skew}(#1)}

\DeclareMathOperator*{\argmin}{arg\,min}

\newcommand{\parenths}[1]{\left( #1 \right)}
\newcommand{\brackets}[1]{\left[ #1 \right]}

\author{Ehsan Abedi, Simone Carlo Surace
\thanks{E. Abedi performed this research during a semester project at University of Bern, Switzerland, while enrolled in the Master Program in Computational Science and Engineering at EPFL, Lausanne, Switzerland, and was supported by the Vahabzadeh Foundation.
        {\tt\small ehsan.abedi@epfl.ch}}%
\thanks{S.C. Surace is with the Department of Physiology, University of Bern, Switzerland, and is supported by the Swiss National Science Foundation, grant PP00P3\_179060.
        {\tt\small surace@pyl.unibe.ch}}%
}

\begin{document}
\bstctlcite{MyBSTcontrol} 

\maketitle
\thispagestyle{empty}
\pagestyle{empty}

\begin{abstract}

Feedback particle filters (FPFs) are Monte-Carlo approximations of the solution of the filtering problem in continuous time.
The samples or particles evolve according to a feedback control law in order to track the posterior distribution.
However, it is known that by itself, the requirement to track the posterior does not lead to a unique algorithm.
Given a particle filter, another one can be constructed by applying a time-dependent transformation of the particles that keeps the posterior distribution invariant.
Here, we characterize this gauge freedom within the class of FPFs for the linear-Gaussian filtering problem, and thereby extend previously known parametrized families of linear FPFs.
\end{abstract}

\section{Introduction}
\label{intro}

The filtering problem is the problem of estimating a quantity evolving in time that is accessible only through partial and noisy observations.
This is commonly formalized as the problem of finding the conditional distribution of the hidden state at time $t$ given all observations up to that time.
Despite its long history and rich developments around its theory (see \cite{Bain2009}, Section 1.3), the principal challenge of implementing filters in practical applications centers around the lack of closed-form solutions and the resulting necessity to find efficient and robust numerical approximations.

The issue has become even more severe in the era of big data where the dimensionality of the processes is very large.
The main approach to approximate the conditional distribution, sequential Monte-Carlo or particle filtering (see \cite{Doucet2011} for a survey and pointers to the literature), is known to exhibit a curse of dimensionality as the number of dimensions of the observations grows \cite{Daum2003}-\nocite{Snyder2008,Rebeschini2015}\cite{Surace2019}.
The problem can be traced down to the use of importance weights and their increasing degeneracy as time progresses (see \cite{Surace2019} and the references therein).

More recently, particle filters without importance weights \cite{Evensen2003}-\nocite{Bergemann2012,Yang2013,Yang2016}\cite{Daum2010} have been gaining attention. 
While lacking the principal vulnerability of weighted particle filters, many theoretical questions remain open \cite{Laugesen2015}-\nocite{Daum2016}\cite{DelMoral2018}. 
One question that has recently received some interest is the non-uniqueness of the law of the process approximating the filtering distribution.
It has been shown \cite{Taghvaei2016}-\cite{Kim2018} that at least in the linear-Gaussian case there are many ways to construct particle dynamics -- both deterministic and stochastic -- that track the exact posterior distribution (given by the Kalman-Bucy filter).

Here, we systematically investigate the degrees of freedom in choosing dynamics of the particles (within some constraints) for the linear-Gaussian filtering problem while keeping their distribution aligned with the exact conditional distribution (assuming that the initial distribution of the filter is Gaussian).
We characterize a group of transformations (which we call gauge transformations, taking inspiration by similar transformations appearing in theoretical physics) that preserve the conditional distribution and describe how it acts on a class of linear feedback particle filters that extends the ones in the literature.
Moreover, we propose a cost function on the family of all such filters.
We identify a known and a new feedback particle filter as a (constrained) optimum of this cost function.

The remainder of the paper is structured as follows: 
in Section~\ref{prelim}, we discuss the filtering problem in general and the linear-Gaussian case that is the focus of this paper in particular. 
We define the notion of particle filter that is the focus of this work and provide an overview of previous work regarding such filters and the history of the non-uniqueness problem.
In Section~\ref{symmetries}, we discuss the symmetries (gauge transformations) of the problem, define certain classes of filters on which the symmetries act, and describe this action in detail.
In Section~\ref{speccase} we show how existing linear feedback particle filters arise as special cases of the class of linear feedback particle filters defined earlier.
In Section~\ref{optimality}, we introduce an optimality criterion and optimize it with and without constraints, obtaining two specific types of feedback particle filters. 
Lastly, in Section~\ref{discussion}, we discuss the implications of our results and comment on future directions.

\section{Preliminaries and background}
\label{prelim}

The classical filtering problem is to find the conditional distribution of the hidden state $X_t$ given observations $\mathscr{F}^Y_t$ for the stochastic system given by
\begin{align}
dX_t&=f(X_t)dt+g(X_t)dW_t, \\
dY_t&=h(X_t)dt+dV_t,
\end{align}
where $X_t$ and $Y_t$ are valued in $\mathds{R}^n$ and $\mathds{R}^m$ respectively, $f,g,h$ are (known) vector- or matrix-valued functions satisfying suitable conditions for the well-posedness of the stochastic differential equations (SDEs), and $W_t$ and $V_t$ are independent $n'$ and $m$-dimensional Brownian motions respectively.
The distribution of $X_0$ is assumed to be known and independent of the Brownian motions.

\subsection{Linear-Gaussian filtering problem}
In this paper, we will focus on the special linear-Gaussian case of the above problem, where $f,g$, and $h$ are chosen such that
\begin{align}
dX_t&=A X_t dt+B dW_t, \\
dY_t&=C X_tdt+dV_t,
\end{align}
for some $A\in\R^{n\times n}$, $B\in\R^{n\times n'}$, and $C\in\R^{m\times n}$, and
$X_0$ has Gaussian distribution. 
This filtering problem has an exact solution, called the Kalman-Bucy filter \cite{Kalman1961}.
The conditional distribution of $X_t$ given $\mathscr{F}^Y_t$ is multivariate Gaussian with mean $\mu_t$ and covariance matrix $P_t$ which jointly evolve as
\begin{align}
d\mu_t&=A\mu_t dt+P_tC^{\top}(dY_t-C\mu_tdt), \\
dP_t&=(BB^{\top}+AP_t+P_tA^{\top}-P_tC^{\top}CP_t)dt,
\end{align}
where $\mu_0$ and $P_0$ are chosen according to the distribution of $X_0$.
For convenience, throughout this article it will be assumed that $P_0$ is strictly positive definite.
Under this assumption, $P_t$ remains strictly positive definite for all $t\geq 0$ (see Proposition 1.1 in \cite{Dieci1994}).

\subsection{Particle filters}
In the context of this paper, a particle filter is any approximation of the conditional distribution of $X_t$ given $\mathscr{F}^Y_t$ by a set of (unweighted) samples or particles $S_t^{(i)}$ for $i=1,...,N$, such that $S_t^{(i)}$ are $\mathscr{F}^{Y,Z}_t$-adapted processes.\footnote{This use of the term `particle filter' differs of the usual one, where the samples are weighted by importance weights.
In this paper, we shall only be concerned with unweighted particle filters, which are also known as interacting particle systems or ensemble Kalman filters \cite{Crisan1999,Evensen2003}.}
Here, $\mathscr{F}^{Y,Z}_t$ is the filtration generated by the process $(Y_t,Z_t)$, where $Z_t$ is some process independent from $\mathscr{F}^{X,Y}_t$ (for example, some additional noise in the particle dynamics).
In the following, we will only consider `symmetric' particle filters for which all particles have the same conditional distribution given $\mathscr{F}^Y_t$.
We will therefore talk about the dynamics of a single representative particle $S_t$, omitting the particle index $i$.
The distribution of $S_t$ will still depend on the number of particles $N$.
Such a particle filter is called asymptotically exact if the distribution of $S_t$ given $\mathscr{F}^Y_t$ converges to the conditional distribution of $X_t$ as $N\to\infty$.

\subsection{Feedback particle filter (FPF)}
\label{poissonsect}
An asymptotically exact filter has been found in \cite{Yang2013} and \cite{Yang2016} based on mean-field optimal control.
It is usually referred to simply as \emph{Feedback Particle Filter}, but in order to distinguish it from similar algorithms, we will refer to this specific algorithm as stochastic feedback particle filter (sFPF). 
Specifically, the sFPF is derived by finding control terms $u$ and $K$ for the particle dynamics\footnote{The FPF is more naturally expressed in Stratonovich form (notation: $\circ$).}
\begin{equation}
dS_t=f(S_t)dt+g(S_t)d\bar W_t+u_t(S_t)dt+K_t(S_t)\circ dY_t,
\label{controllaw}
\end{equation}
such that the particle filter whose particles $S_t^{(i)}$ evolve according to the SDE \eqref{controllaw} is asymptotically exact.
Here, $\bar W_t$ is an $n'$-dimensional Brownian motion that is independent of $\mathscr{F}^{X,Y}_t$ and plays the role of $Z_t$ in the previous paragraph.
The sFPF can be derived by `aligning' the Fokker-Planck equation of the particle filter with the Kushner-Stratonovich filtering equation in an appropriate sense (see \cite{Yang2016}, Section 2), which leads to a McKean-Vlasov or mean-field equation where $u_t$ and $K_t$ depend on the distribution of $S_t$. 
As a result, the FPF is exact in the mean-field sense if initialized with the correct initial distribution.
However, the computation of the gain term $K$ requires the solution of a linear boundary value problem (BVP) at each instant in time, which accounts for the bulk of the computational cost this algorithm.

Specifically, the $j$'th column of the gain matrix $K$ is any vector-valued function that satisfies a weighted Poisson equation
\begin{equation}
\nabla\cdot(p(x)K^{.j}(x)) = - \parenths{h^j(x) - \hat h^j}p(x)
\label{gaineq}
\end{equation}
with suitable boundary conditions, where $p$ is the density of the current particle distribution and $\hat h^j=\int h^j(y)p(y)dy$.
This equation does not have uniqueness of solutions, as any solution $K$ may be perturbed by a divergence-free vector field $K'$, i.e. $\nabla\cdot(pK') = 0$ to produce a new solution $K+K'$.
Uniqueness can be obtained by requiring $K$ to be a gradient\footnote{Under suitable conditions on $p$, this gradient solution can be interpreted as the minimum-energy gain, which is related to the dynamic (Benamou-Brenier) formulation of optimal transport (see Lemma 8.4.2 in \cite{Ambrosio2008}).}, i.e. $K=\nabla\phi$.
Whatever $K$ is chosen, the general solution for $u$ can be written as $u=-\tfrac{1}{2}(h+\hat h)K + u'$, where $u'$ is divergence-free in the sense above.

In the linear-Gaussian case, the gradient-form solution of the BVP \eqref{gaineq} can be obtained in closed form.
It is given by the Kalman gain $K_t\equiv P_tC^{\top}$ (see \cite{Yang2016}, Section 4.2).
By choosing $u_t=-\tfrac{1}{2}(h+\hat h)K_t$, the resulting FPF -- which will henceforth be referred to as the \emph{stochastic linear FPF} (slFPF) -- is equivalent to the square-root form of the Ensemble Kalman-Bucy filter.
It takes the form
\begin{equation}
dS_t=AS_tdt+P_tC^{\top}\parenths{dY_t-\tfrac{1}{2}C(S_t+\mu_t)dt}+Bd\bar W_t.
\label{slFPF}
\end{equation}
If $S_0$ is normally distributed with mean $\mu_0$ and covariance matrix $P_0$, the conditional distribution of $S_t$ given $\mathscr{F}^Y_t$ has mean and covariance $\mu_t$ and $P_t$ respectively for all $t\geq 0$.

\subsection{Non-uniqueness of the control law}
\label{nonunique}
The linear case allows us to explore the non-uniqueness inherent in the choice of $K$ and $u$.
In the linear case, a larger class of solutions of Eq.~\eqref{gaineq} can be obtained by adding a linear divergence-free field to each column of $K$, i.e. setting $K^{.j}_t(x)=P_t(C^{j.}+\Pi^{(j)}_t (x-\mu_t))$, where $\Pi^{(j)}_t$ are skew-symmetric matrices. 
In addition, one may choose $u_t(x)=-\tfrac{1}{2}K_t(x)C(x+\mu_t)+P_t\tilde\Pi_t(x-\mu_t)$, where $\tilde\Pi_t$ is yet another skew-symmetric matrix.

Besides the modification of $K$ and $u$ by divergence-free terms, there are other modifications of the slFPF dynamics \eqref{slFPF} for which the distribution of $S_t$ agrees with the Kalman-Bucy filter.
In \cite{Taghvaei2016} a deterministic linear FPF (deterministic refers to the fact that there is no independent noise term) was derived from an optimal transport perspective.
We will refer to it as \emph{optimal transport deterministic linear feedback particle filter} (OTdetFPF). 
Its dynamics are given by  
\begin{multline}
dS_t=AS_tdt+P_tC^{\top}\parenths{dY_t-\tfrac{1}{2}C(S_t+\mu_t)dt}\\
+\hat\Omega_tP_t^{-1}(S_t-\mu_t)dt+\frac{1}{2}BB^{\top}P_t^{-1}(S_t-\mu_t)dt,
\label{OTdetFPF}
\end{multline}
where $\hat\Omega_t$ is the unique skew-symmetric matrix satisfying
\begin{multline}
\hat\Omega_tP_t^{-1}+P_t^{-1}\hat\Omega_t=A^{\top}-A\\
+\tfrac{1}{2}\Big(P_t^{-1}BB^{\top}-BB^{\top}P_t^{-1}\\
+P_tC^{\top}C-C^{\top}CP_t\Big).
\label{OTOmega}
\end{multline}
The OTdetFPF replaces the Brownian motion term of the sFPF by a repulsive term that drives the particles away from their mean. 
In addition, it has a skew-symmetric term that arises from the minimization of the transportation cost.
Since the skew-symmetric term does not affect the distribution, a parametrized family of deterministic FPFs (detFPF) may be constructed by choosing it arbitrarily, in particular by setting it to zero.

In another paper \cite{Kim2018}, an approach based on optimal control theory and duality formalisms is used to extend the linear FPF to a two-parameter family of FPFs.
In the next section, all of the previously mentioned transformations will be captured by a common parametrization, see in particular Section~\ref{speccase}.

\subsection{Notations and definitions}
We denote by $S$ the process $(S_t)_{t\geq 0}$ as a random variable with values in a suitable subspace of functions $\R\to\R^n$ (for most cases, the space of continuous functions will suffice), with its law given by a probability measure on the corresponding Borel $\sigma$-algebra.
If the conditional distribution of $S_t$ given $\mathscr{F}^Y_t$ agrees with the conditional distribution of $X_t$ given $\mathscr{F}^Y_t$ for all $t\geq 0$, such a process is called \emph{particle filter}\footnote{We only consider exact particle filters here. The idea is that fore finite $N$ the mean-field quantities appearing in the control law are replaced by empirical estimates, yielding an asymptotically exact filter (although the details are nontrivial, see e.g. \cite{Mandel2011})}.
If $S$ is adapted to $\mathscr{F}^{Y,Z}_t$, where $Z$ consists of $r\geq 0$ Brownian motions independent of $Y$ and if $S$ is the solution to an SDE, it will be called \emph{feedback particle filter}.
If $r=0$ it is called \emph{deterministic feedback particle filter}.
We denote by $\Skew{n}$ the real vector-space of skew-symmetric $n\times n$-matrices, i.e. of matrices $X$ such that $X^{\top}=-X$, where $\cdot^{\top}$ denotes transposition.
The notation $\text{GL}(n)$ is used for the multiplicative group of invertible $n\times n$-matrices.

\section{The class of linear particle filters and their symmetries}
\label{symmetries}
In the linear-Gaussian case, a particle filter is exact if the conditional distribution of $S_t$ given $\mathscr{F}^Y_t$ is multivariate Gaussian with mean equal to $\mu_t$ and covariance matrix equal to $P_t$ (the mean and covariance matrix of the Kalman-Bucy filter).
If we split off the mean and define $E_t:=S_t-\mu_t$, the process $E_t$ can be any Gaussian process with conditional mean equal to zero and conditional covariance function $k(t,t')$ such that $k(t,t)=P_t$ for all $t\geq 0$. 
In this section, we first look at a group of linear time-dependent transformations that preserve this structure, called \emph{gauge transformations}.
Then we consider the subclass $\mathfrak{F}$ of filters for which $E_t$ is the solution to an SDE, which we will call the class of \emph{linear feedback particle filters}, and the subclass $\mathfrak{F}_{\text{det}}\subset \mathfrak{F}$ for which $E_t$ does not have an independent source of noise, called \emph{deterministic linear feedback particle filters}.
We then describe the action of the group on $\mathfrak{F}$, proving that the group of gauge transformations acts transitively on $\mathfrak{F}_{\text{det}}$.

\subsection{Gauge freedom of linear particle filters}
From any given particle filter $S_t=\mu_t+E_t$, another particle filter $\tilde S_t=\mu_t+\tilde E_t$ can be constructed by applying a linear transformation $\tilde E_t=g_tE_t$, where $g_t\in\text{GL}(n)$ is adapted to $\mathscr{F}^Y_t$.
By construction, the conditional means of $\tilde S_t$ and $S_t$ agree, and if in addition $g_t$ satifies
\begin{equation}
g_tP_tg_t^{\top}=P_t, \quad t\geq 0, 
\label{group}
\end{equation}
the conditional variances of $\tilde S_t$ and $S_t$ also agree.

For fixed $t$, the subset consisting of those $g\in\text{GL}(n)$ that satisfy Eq.~\eqref{group} forms a (random and time-dependent) Lie subgroup of $\text{GL}(n)$, which we denote by $\mathcal{G}_t$. 
Its Lie algebra $\mathfrak{g}_t$ consists of all matrices of the form $\Omega P_t^{-1}$, where $\Omega\in\Skew{n}$ is arbitrary.

In principle, the choice of the function $g:[0,\infty)\to GL(n)$, $t\mapsto g_t$ does not have to be regular.
For example, after simulating a sample of particles up to time $t$, a sample for $s\leq t$ can be modified by a transformation $g\in\mathcal{G}_s$ without concern for samples at other times as long as only information up to time $s$ is used.
However, in the following we will restrict to choices of $g$ with more regularity.

For example, $g$ may be chosen to be a continuously differentiable function governed by an ordinary differential equation (ODE) $\dot g_t=M_tP_t^{-1}g_t$.
This parametrization is convenient because Eq.~\eqref{group} implies the constraint $M_t=\tfrac{1}{2}(\dot P_t-g_t\dot P_tg_t^{\top})+\Upsilon^{(0)}_t$, where $\Upsilon^{(0)}$ is a continuous but otherwise arbitrary function with values in $\Skew{n}$.
The skew-symmetric component of $M_t$ accounts for motion of $g_t$ along the group $\mathcal{G}_t$, whereas the symmetric component is due to the change in $\mathcal{G}_t$ itself. 

More generally, $g_t$ may be given by the solution of an SDE involving the observations, i.e.
\begin{equation}
dg_t=\parenths{M_tdt+\sum_{i=1}^m\Upsilon^{(i)}_tdY^{(i)}_t}P_t^{-1}g_t,
\label{gSDE}
\end{equation}
where $Y^{(i)}_t$ denotes the component number $i$ of $Y_t$.
By differentiating Eq.~\eqref{group} and matching terms, we find that $\Upsilon^{(i)}_t\in\Skew{n}$, $i=1,...,m$ can be chosen arbitrarily, whereas $M_t$ has to satisfy the constraint
\begin{equation}
M_t=\frac{1}{2}\parenths{\dot P_t-g_t\dot P_tg_t^{\top}-\sum_{i=1}^m\Upsilon^{(i)}_tP^{-1}_t(\Upsilon^{(i)}_t)^{\top}}+\Upsilon^{(0)}_t,
\label{Mconstr}
\end{equation}
where $\Upsilon^{(0)}_t\in\Skew{n}$ is again arbitrary.
We call any $g$ of the above form \eqref{gSDE}-\eqref{Mconstr} a \emph{deterministic gauge transformation}.

\subsection{A general class of linear feedback particle filters}
Suppose that $S_t=\mu_t+E_t$, where the dynamics of $E_t$ are given by
\begin{equation}
dE_t=\parenths{G_tdt+\sum_{i=1}^m\Omega^{(i)}_tdY^{(i)}_t}P_t^{-1}E_t+\sum_{j=1}^rH^{(j)}_td\bar W^{(j)}_t,
\label{Edyn}
\end{equation}
where $r\in\{0,1,2,...\}$ (for $r=0$ the sum over $j$ is zero by convention), $\bar W^{(j)}_t$ are scalar Brownian motions independent of each other and of $\mathscr{F}^{X,Y}_t$, and $E_0$ has zero mean and covariance matrix equal to $P_0$.
In addition, all coefficients are assumed to be adapted to $\mathscr{F}^Y_t$ and bounded.
By construction, $S_t$ has conditional mean equal to $\mu_t$.
Moreover, $S_t$ has conditional covariance equal to $P_t$ for all $t\geq 0$ if and only if $\Omega^{(i)}_t\in\Skew{n} $, $i=1,...,m$, and $G_t$ satisfies the constraint
\begin{multline}
G_t=\frac{1}{2}\Bigg(\dot P_t+\sum_{i=1}^m\Omega^{(i)}_tP_t^{-1}\Omega^{(i)}_t
\\
-\sum_{j=1}^rH^{(j)}_t(H^{(j)}_t)^{\top}\Bigg)+\Omega^{(0)}_t,
\label{G0}
\end{multline}
where $\Omega^{(0)}_t\in \Skew{n}$ is arbitrary.

\begin{definition}
The class of all processes $S_t=\mu_t+E_t$ with $E_t$ having dynamics according to Eqs.~\eqref{Edyn}-\eqref{G0} where $\Omega^{(i)}_t$, $i=0,1,...,m$ are $\mathscr{F}^Y_t$-adapted $\Skew{n}$-valued processes, is denoted by $\mathfrak{F}$ and called the class of \emph{linear feedback particle filters}.
The subclass $\mathfrak{F}_{\text{det}}\subset \mathfrak{F}$ for which $r=0$ is called \emph{deterministic linear feedback particle filters}.
\end{definition}

\subsection{Special cases of linear feedback particle filters}
\label{speccase}
First, it is worth pointing out that the filters obtained by adding linear divergence-free vector fields to $K$ and $u$ (as described in the beginning of Section~\ref{nonunique}) belong to $\mathscr{F}$, as they can be related to the above parametrization by setting $\Omega^{(i)}_t=P_t\Pi^{(i)}_tP_t$ and $\Omega^{(0)}_t=\tfrac{1}{2}AP_t-\tfrac{1}{2}P_tA-P_t\tilde\Pi_tP_t-\tfrac{1}{2}\sum_{i=1}^mC^{i.}\mu_t P_t\Pi^{(i)}_tP_t$.

By adding the mean dynamics ($d\mu_t$) to $dS_t$ from Eqs.~\eqref{Edyn}-\eqref{G0}, we obtain the dynamics of particle filters in $\mathfrak{F}$:
\begin{multline}
dS_t=AS_tdt+P_tC^{\top}\parenths{dY_t-\tfrac{1}{2}C(S_t+\mu_t)dt}\\
+\tilde\Omega_tP_t^{-1}(S_t-\mu_t)dt+\frac{1}{2}BB^{\top}P_t^{-1}(S_t-\mu_t)dt\\
+\sum_{j=1}^rH^{(j)}_td\bar W^{(j)}_t-\frac{1}{2}\sum_{j=1}^rH^{(j)}_t(H^{(j)}_t)^{\top}P_t^{-1}(S_t-\mu_t)dt,
\label{generaldS}
\end{multline}
where $\tilde\Omega_t=\tfrac{1}{2}\Omega^{(0)}_t-\tfrac{1}{2}(AP_t-P_tA^{\top})+\tfrac{1}{2}\sum_{i=1}^m\Omega^{(i)}_tP_t^{-1}\Omega^{(i)}_t$.

Note that although $\tilde\Omega_t$ is comprised of several terms, the presence of the arbitrary skew-symmetric matrix $\Omega^{(0)}_t$ means that there are no constraints on its time evolution.
We may thus split it as $\tilde\Omega_t=\hat\Omega_t+(\tilde\Omega_t-\hat\Omega_t)$, where $\hat\Omega_t$ is given by the solution to Eq.~\eqref{OTOmega}.
Thus, for $r=0$ (or $r>0$ and $H^{(j)}_t=0$, $j=1,...,r$) and $\tilde\Omega_t-\hat\Omega_t=0$, Eq.~\eqref{generaldS} takes the form of the OTdetFPF from Eq.~\eqref{OTdetFPF}. 
If we set $r=n'$, $H^{(j)}_t=j\text{'th column of }B$, and $\tilde\Omega_t=0$, the last term in the second line of \eqref{generaldS} cancels with the last term in the third line, and we obtain the slFPF from Eq.~\eqref{slFPF}.
As a last example, if $r=n'+m$, $H^{(j)}_t=\gamma_1\cdot j\text{'th column of }B$ for $j=1,...,n'$, and $H^{(n'+j)}_t=\gamma_2\cdot j\text{'th column of }P_tC^{\top}$ for $j=1,...,m$, we obtain the two-parameter family with parameters $\gamma_1,\gamma_2$ from \cite{Kim2018}.

\subsection{Action of gauge transformations on the class of linear feedback particle filters}
The following two results summarize the details of how gauge transformations act upon the classes $\mathfrak{F}$ and its subclasses defined above.
\begin{proposition}
Let $S\in\mathfrak{F}$ and $\tilde S_t=\mu_t+\tilde E_t$, where $\tilde E_t=g_tE_t$ with $g$ a deterministic gauge transformation according to Eqs.~\eqref{gSDE}-\eqref{Mconstr}.
Then $\tilde S$ also belongs to $\mathfrak{F}$ with $\tilde \Omega^{(i)}_t$, $i=0,1,...,m$ and $\tilde H^{(j)}_t$, given by
\begin{align}
\label{Omegatilde0}\tilde \Omega^{(0)}_t&=g_t\Omega_tg_t^{\top}+\Upsilon^{(0)}_t\nonumber \\
&\quad +\frac{1}{2}\sum_{i=1}^m\parenths{\Upsilon^{(i)}_tP_t^{-1}\tilde \Omega^{(i)}_t-\tilde \Omega^{(i)}_tP_t^{-1}\Upsilon^{(i)}_t},\\
\label{Omegatildei}\tilde \Omega^{(i)}_t&=g_t\Omega^{(i)}_tg_t^{\top}+\Upsilon^{(i)}_t, \quad i=1,..,m, \\
\label{Htildej}\tilde H^{(j)}_t&=g_tH^{(j)}_t, \quad j=1,...,r.
\end{align}
\end{proposition}
\begin{proof}
We have $d\tilde E_t=dg_tE_t+g_tdE_t+dg_tdE_t$.
By using Eqs.~\eqref{gSDE} and \eqref{Edyn}, we obtain
\begin{align}
dg_tE_t&=\parenths{M_tdt+\sum_{i=1}^m\Upsilon^{(i)}_tdY^{(i)}_t}P_t^{-1}\tilde E_t,\\
g_tdE_t&=g_t\parenths{G_tdt+\sum_{i=1}^m\Omega^{(i)}_tdY^{(i)}_t}P_t^{-1}E_t\nonumber \\
&\quad+\sum_{j=1}^rg_tH^{(j)}_td\bar W^{(j)}_t, \\
dg_tdE_t&=\sum_{i=1}^m\Upsilon^{(i)}_tP_t^{-1}g_t\Omega^{(i)}_tP_t^{-1}E_t dt.
\end{align}
By rewriting $P_t^{-1}E_t=g_t^{\top}g_t^{-\top}P_t^{-1}g_t^{-1}\tilde E_t$, noting that $g_t^{-\top}P_t^{-1}g_t^{-1}=P_t^{-1}$, and then collecting all the terms, we obtain Eqs.~\eqref{Omegatildei} and \eqref{Htildej} as well as $\tilde G_t=g_tG_tg_t^{\top}+M_t$.
By using Eqs.~\eqref{Mconstr} and \eqref{G0}, we obtain
\begin{multline}
g_tG_tg_t^{\top}+M_t=\frac{1}{2}g_t\dot P_tg_t^{\top}+\frac{1}{2}\sum_{i=1}^mg_t\Omega^{(i)}_tP_t^{-1}\Omega^{(i)}_tg_t^{\top}\\
-\frac{1}{2}\sum_{j=1}^rg_tH^{(j)}_t(g_tH^{(j)}_t)^{\top}+g_t\Omega^{(0)}_tg_t^{\top}\\
+\frac{1}{2}\dot P_t-\frac{1}{2}g_t\dot P_tg_t^{\top}-\frac{1}{2}\sum_{i=1}^m\Upsilon^{(i)}_tP^{-1}_t(\Upsilon^{(i)}_t)^{\top}+\Upsilon^{(0)}_t.
\end{multline}
By rewriting $P_t^{-1}=g_t^{\top}P_t^{-1}g_t$ and then using Eq.~\eqref{Omegatildei} to substitute $g_t\Omega^{(i)}_tg_t^{\top}$, after cancelling all terms we obtain
\begin{multline}
g_tG_tg_t^{\top}+M_t=\frac{1}{2}\Big(\dot P_t+\sum_{i=1}^m\tilde\Omega^{(i)}_tP_t^{-1}\tilde\Omega^{(i)}_t\\
-\sum_{j=1}^r\tilde H^{(j)}_t(\tilde H^{(j)}_t)^{\top}\Big)+g_t\Omega^{(0)}_tg_t^{\top}+\Upsilon^{(0)}_t,
\end{multline}
from which Eq.~\eqref{Omegatilde0} follows. 
\end{proof}

\begin{corollary}
The deterministic gauge transformations act transitively on $\mathfrak{F}_{\text{det}}$, i.e. every deterministic gauge transformation of an $S\in\mathfrak{F}_{\text{det}}$ is in $\mathfrak{F}_{\text{det}}$, and for any pair of $S,\tilde S\in\mathfrak{F}_{\text{det}}$ there is a deterministic gauge transformation that maps $S$ to $\tilde S$.
\end{corollary}

\section{Optimality criteria}
\label{optimality}

We will now return to the multidimensional case.
In order to select a particle filter among all linear feedback particle filters, additional criteria are required.
Since every filter $S\in\mathfrak{F}$ is specified by a choice of $\Omega^{(i)}_t\in\Skew{n}$ for $i=0,1,...,m$ and a vector $H^{(j)}_t\in\R^n$ for $j=1,...,r$ and for all $t$, an optimality criterion can be formulated as a function of these quantities.
Define 
\begin{equation}
\mathscr{L}_{P,r}(\Omega^{(0)},...,\Omega^{(m)},H^{(1)},...,H^{(r)})=\text{tr}\, GP^{-1}G^{\top},
\label{costfunction}
\end{equation}
where $\text{tr}$ denotes the trace operator and $G$ is given by Eq.~\eqref{G0}. 
This term appears to quadratic order in $dt$ when expanding $\text{tr}\,\mathds{E}[dE_tdE_t^{\top}|\mathscr{F}^Y_t]$ and can be associated with the transport cost as in \cite{Taghvaei2016}.
It is therefore unsurprising that its minimization (under the constraint that the observation and Brownian motion terms are absent) yields the OTdetFPF from \cite{Taghvaei2016}.

\begin{proposition}
The filter given by $\Omega^{(i)}_t=0$, $i=1,..,m$, $H^{(j)}=0$, $j=1,...,r$, and
\begin{equation}
\Omega^{(0)}_t=\Omega^{\ast}_t=\argmin_{\Omega\in\Skew{n}}\mathscr{L}_{P_t,0}(\Omega,0,...,0)
\end{equation}
is identical to the optimal transport deterministic linear feedback particle filter (OTdetFPF).
\label{PropOT}
\end{proposition}
\begin{proof}
The OTdetFPF has dynamics\footnote{The reference \cite{Taghvaei2016} assumed $B=\mathds{1}$. 
For this reason, a few additional terms involving $B$ appear here.}
\begin{equation}
dE_t=\parenths{A+\tfrac{1}{2}BB^T-\tfrac{1}{2}P_tC^{\top}CP_t+\hat\Omega_t}P_t^{-1}E_tdt,
\end{equation}
where $\hat\Omega_t$ is the unique solution of Eq.~\eqref{OTOmega}.
This corresponds to Eq.~\eqref{Edyn} with $\Omega^{(i)}_t=0$, $i=1,..,m$, $H^{(j)}=0$, $j=1,...,r$, and
\begin{equation}
\Omega^{(0)}_t=\tfrac{1}{2}\parenths{AP_t-P_tA^{\top}+\hat\Omega_t}.
\end{equation}
Any critical point of the function $\mathscr{L}_{P,0}(\Omega,0,...,0)$ satisfies 
\begin{equation}
\text{tr}\brackets{X\parenths{P^{-1}G^{\top}-GP^{-1}}}=0, \quad\forall X\in\Skew{n}.
\end{equation}
Since $(X,Y)\mapsto\text{tr}\,XY$ defines an inner product on $\Skew{n}$, this implies $P^{-1}G^{\top}-GP^{-1}=0$. 
Substituting $G=\tfrac{1}{2}\dot P+\Omega$, we obtain the equation
\begin{multline}
\Omega P^{-1}+P^{-1}\Omega=\tfrac{1}{2}\parenths{P^{-1}\dot P-\dot PP^{-1}}\\
=\tfrac{1}{2}\Big(P^{-1}BB^{\top}-BB^{\top}P^{-1}+P^{-1}AP+A^{\top}\\
-A-PA^{\top}P^{-1}+PC^{\top}C-C^{\top}CP\Big).
\end{multline}
It can be checked that the unique solution is given by $\Omega=\tfrac{1}{2}(AP-PA^{\top}+\hat\Omega)$.
\end{proof}
Whereas the optimal transport formulation in \cite{Taghvaei2016} always yields deterministic filters, the cost function \eqref{costfunction} may be minimized without the constraints of Proposition~\ref{PropOT}.
This yields a new class of filters that has not yet explicitly appeared in the literature.

\begin{proposition}
Let $S$ be a particle filter specified by data $\mathscr{D}_t=(\Omega^{(0)}_t,\Omega^{(1)}_t,...,\Omega^{(m)}_t,H^{(1)}_t,...,H^{(r)}_t)$ consisting of $\Omega^{(i)}_t\in\Skew{n}$ for $i=0,1,...,m$ and a vector $H^{(j)}_t\in\R^n$ for $j=1,...,r$ for some $r\geq 0$.
The following properties are equivalent:
\begin{enumerate}[i)]
\item $\mathscr{L}_{P_t,r}$ has a critical point at $\mathscr{D}_t$,
\item $\mathscr{L}_{P_t,r}$ has a global minimum at $\mathscr{D}_t$,
\item $\mathscr{D}_t$ is given by $\Omega^{(0)}_t=0$ and 
\begin{equation}
\sum_{j=1}^rH^{(j)}_t(H^{(j)}_t)^{\top}-\sum_{i=1}^m\Omega^{(i)}_tP_t^{-1}\Omega^{(i)}_t=\dot P_t.
\end{equation}
\end{enumerate}
\label{PropOTwoC}
\end{proposition}
\begin{proof}
By differentiating $\mathscr{L}_{P,r}$, we have i) if and only if 
\begin{multline}
0=\text{tr}\brackets{K^{(0)}\parenths{P^{-1}G^{\top}-GP^{-1}}}\\
+\sum_{i=1}^m\text{tr}\brackets{K^{(i)}P^{-1}\Omega^{i}\parenths{P^{-1}(G)^{\top}+GP^{-1}}}\\
-\sum_{j=1}^m\text{tr}\brackets{l^{(j)}(H^{(j)})^{\top}\parenths{P^{-1}(G)^{\top}+GP^{-1}}}
\end{multline}
for all $K^{(i)}\in\Skew{n}$ and all $l^{(j)}\in\R^n$, $i=0,...,m$ and $j=1,...,r$.
Since $\text{tr}$ is an inner product, this is equivalent to
\begin{align}
P^{-1}(G)^{\top}-GP^{-1}&=0, \\
P^{-1}(G)^{\top}+GP^{-1}&=0,
\end{align}
which in turn is equivalent to $G=0$ and hence to iii).
This shows i)$\Leftrightarrow$iii)
Moreover, $G=0$ when plugged into $\mathscr{L}_{P,r}$ gives a value of $0$, which is the global minimum for this function.
This shows that iii)$\Rightarrow$ii), which concludes the proof.
\end{proof}

\begin{remark}
For $n=1$, since all skew-symmetric $1\times 1$-matrices are zero, condition iii) in Proposition~\ref{PropOTwoC} reads $\sum_{j=1}^r (h^{(j)}_t)^2=b^2+2aP_t-c^2P_t^2$.
This has a solution if and only if $P_0\leq P_t\leq \frac{a+\sqrt{a^2+b^2c^2}}{c^2}$. 
Otherwise there are no global minimizers of $\mathscr{L}_{P_t,r}$.
\end{remark}

\section{Discussion}
\label{discussion}

In this paper, we studied the intrinsic freedom in choosing dynamics of a particle filter if the only requirement for it is that the distribution of $S_t$ match the exact posterior distribution.
We studied this for the example of the linear-Gaussian filtering problem where all calculations can be done explicitly, and the transformations can be assumed to be linear.

As noted in Section~\ref{nonunique}, part of this freedom is clearly present in the nonlinear case in terms of the freedom of modifying any given pair of solutions for $K$ and $u$ by divergence-free vector fields.
However, since the distribution is not known, such divergence-free fields cannot be easily found.
Other forms of freedom, such as the exchange of noise terms for deterministic terms, that are easily accomplished in the linear case, also present more difficulties in the nonlinear case since it is not known, without having detailed knowledge about the distribution, how to compensate a given diffusion term by a deterministic one.

Generally speaking, in the nonlinear case both the gain estimation problem as well as the problem of finding gauge transformations are equally difficult because the distribution is unknown, and cannot be easily reduced to mean and covariance matrix such as in the linear case.
However, the presence of gauge transformations suggests alternative formulation of the gain estimation problem. 
In particular, it raises the question whether the gradient solution for $K$, as it is used in the literature, is the most relevant solution.

The freedom in the choice of the gain can be viewed from yet another angle when considering the filtering problem for a hidden process with values in a smooth manifold $M$.
If a feedback particle filter is to be designed on a manifold, there is no canonical riemannian metric and therefore no preferred gradient solution for $K$. 
Part of the gauge freedom of nonlinear feedback particle filters can therefore be attributed to the choice of a riemannian metric on the manifold.
As explained in Section~\ref{poissonsect}, under certain conditions on the distribution (in the case of compact $M$, for all distributions that can be expressed as a smooth volume form) the choice of a riemannian metric on the state space gives rise to the gradient solution as a minimum-energy solution, which can also be interpreted as the solution to the following dynamical optimal transport problem for the 2-Wasserstein metric induced by the riemannian metric: given a direction in the space of probability distributions, prescribed by the filtering equation, what is the corresponding vector field in the state space such that the incremental transportation cost is minimized?

In future research, the connection between this optimal transport principle (for the gain function) and the optimal transport deterministic feedback particle filter should be further explored.
Also, the broader significance of optimal transport principles and the choice of riemannian metric in nonlinear filtering deserves further study.
In particular, desiderata such as numerical stability should be connected, if possible, to optimal transport principles in order to enable a principled design of feedback particle filters. 
 
\section{Acknowledgements}

We thank Jean-Pascal Pfister, Anna Kutschireiter, Amir-hossein Taghvaei, and Prashant Mehta for helpful discussions, as well as the anonymous referees for their suggestions for improving the manuscript.

\bibliographystyle{IEEEtran}
\bibliography{settings,library}

%

\end{document}